\documentclass[11pt]{amsart}

\usepackage{amsmath,amssymb,amsfonts}

\newtheorem{Le}{Lemma}[section]

\newtheorem{Th}[Le]{Theorem}

\theoremstyle{definition}
\newtheorem{Def}[Le]{Definition}

\theoremstyle{remark}

\newtheorem{Rem}[Le]{Remark}

\allowdisplaybreaks
\begin{document}

\title{$G_2$-orbifolds from K3 surfaces with ADE-singularities} 
\author{Frank Reidegeld}

\maketitle

\begin{abstract}
We construct compact $G_2$-orbifolds with ADE-singularities that carry 
exactly one parallel spinor. Our examples are related to certain quotients 
of $\mathbb{C}^2\times T^3$ that have been investigated in \cite{Ach}. We shortly 
discuss the physical applications of our examples.   
\end{abstract}

\tableofcontents

\section{Introduction}

In recent years, Riemannian manifolds with holonomy $G_2$ have attracted 
considerable attention. They are studied for purely mathematical reasons and
as compactifications of M-theory. Usually, a $G_2$-manifold is defined 
as a smooth manifold. Nevertheless, it turns out that $G_2$-orbifolds with
ADE-singularities along associative submanifolds are interesting objects, too. 
It is assumed that they arise as boundary components of the moduli space
of smooth $G_2$-structures on a fixed manifold \cite{Morrison1,Morrison2}. 
Moreover, M-theory compactified on a $G_2$-orbifold with ADE-singularities
yields a super Yang-Mills theory with non-abelian gauge group in the low-energy
limit \cite{Ach,Ach2}. 

Explicit examples of such orbifolds with holonomy $G_2$ are hard to 
construct and one-parameter families of smooth $G_2$-manifolds that
converge to a $G_2$-orbifold with ADE-singularities even more. A more
approachable problem is to search for orbifolds with holonomy $\text{Hol}_0\rtimes
\Delta$ where $\text{Hol}_0\in\{1,Sp(1),SU(3)\}$ is the identity component 
of the holonomy and $\Delta$ is a discrete group such that $\text{Hol}_0\rtimes
\Delta$ acts irreducibly on the tangent space. For example, the case 
$\text{Hol}_0=SU(3)$, $\Delta=\mathbb{Z}_2$ yields so called barely 
$G_2$-manifolds. If we search for orbifolds with a holonomy group of
this kind, our problem becomes much simpler since we can use arguments 
from complex geometry. Moreover, such orbifolds share many features
with orbifolds with holonomy $G_2$, for example they admit exactly one 
parallel spinor. This makes them suitable candidates for compactifications of 
M-theory that provide four-dimensional field theories with $\mathcal{N}=1$. 

In the literature \cite{Ach}, there are examples of flat $G_2$-orbifolds with 
ADE-singularities which can serve as local models for more generic 
$G_2$-orbifolds. More explicitly, let $\Gamma$ be a discrete subgroup of 
$SU(2)$. The quotient $\mathbb{C}^2/\Gamma$ carries a hyper-K\"ahler 
structure that is defined by the flat Hermititan metric and the three linearly 
independent self-dual 2-forms on $\mathbb{C}^2$. The product 
$\mathbb{C}^2/\Gamma\times T^3$, where $T^3$ is the three-dimensional 
flat torus, is divided by a discrete group $H$. $H$ acts on both factors separately 
and preserves the hyper-K\"ahler structure on $\mathbb{C}^2/\Gamma$. 
Therefore, the quotient carries a $G_2$-structure. 

The aim of this article is to use the examples from \cite{Ach} to construct 
compact $G_2$-orbifolds with holonomy of type $Sp(1)\rtimes \Delta$. 
In order to do this, we replace $\mathbb{C}^2/\Gamma$ by a K3 surface 
$S$ with singularities that carries a hyper-K\"ahler metric. $H$ shall act 
on $S\times T^3$ such that 

\begin{enumerate}
    \item $H$ acts on $S$ by isometries,
    
    \item the three-dimensional representation of $H$ that is induced by
    the action of $H$ by pull-backs on the space of all K\"ahler
    forms is the same as in \cite{Ach} and
    
    \item the action of $H$ on $T^3$ is the same as in \cite{Ach}.
\end{enumerate}

We prove that there exists a class of examples with the desired properties.
In the most singular case, $S$ has two singularities of type $E_8$, but 
there are also many examples with milder singularities that will be described
in the course of the article. We shortly discuss the low-energy limit of M-theory
compactified on our $G_2$-orbifolds. The bosonic part of the four-dimensional
field theory contains a Yang-Mills theory with gauge group 
$U(1)^{16-\text{rank}(G)}\times G$ where the group $G$ is determined by the 
singularities. 

Although our model is to simplistic to describe the standard model 
of particle physics - it does not contain chiral fermions -, it has several interesting 
features. Since K3 surfaces are extremely well studied, it should be possible to 
determine further physical quantities, e.g. coupling constants, the superpotential etc., 
more or less explicitly. Moreover, our $G_2$-orbifolds are naturally fibered by 
coassociative K3 surfaces. This fact could make them interesting for studying
the duality between M-theory and heterotic string theory.

\section{$G_2$-orbifolds and their applications in physics}

In this section, we introduce some facts about $G_2$-manifolds and 
-orbifolds that we need to describe the examples of \cite{Ach} in detail. 

\begin{Def}
A 3-form $\phi$ on a 7-dimensional manifold $M$ is called a 
\emph{$G_2$-structure} if for all $p\in M$ there exists a local coframe $e^1,\ldots,
e^7$ on a neighborhood $U$ of $p$ such that

\begin{equation}
\label{G2Form}
\phi|_U =  e^{123} + e^{145} + e^{167} + e^{246} - e^{257} - e^{347} 
- e^{356}\:,  
\end{equation}

where $e^{ijk}$ is defined as $e^i\wedge e^j \wedge e^k$.  
\end{Def} 

\begin{Rem}
The differential form (\ref{G2Form}) has $G_2$ as stabilizer group. 
$GL(7,\mathbb{R})$ acts transitively on the set of all forms on a 7-dimensional
vector space with stabilizer $G_2$. Therefore, we could have defined a 
$G_2$-structure as a 3-form that is stabilized by $G_2$ at each point.   
\end{Rem}

On any manifold with a $G_2$-structure $\phi$ there exists a canonical 
metric $g$ which is determined by the following formula:

\begin{equation}
\label{Metric}
g(X,Y) e^{1234567} = \frac{1}{6} (X\lrcorner \phi) \wedge (Y\lrcorner \phi)
\wedge \phi 
\end{equation}

\begin{Def} 
A $G_2$-structure $\phi$ is called \emph{parallel} if $d\phi = d\ast\phi = 0$,
where the Hodge-star $\ast$ is with respect to the canonical metric and the
volume form that is locally given by $e^{1234567}$. A 7-dimensional manifold 
together with a parallel $G_2$-structure is called a \emph{$G_2$-manifold}. 
\end{Def}

\begin{Le}
Let $(M,\phi)$ be a $G_2$-manifold and let $g$ be the metric on $M$ that is defined
by Equation (\ref{Metric}). Then 

\begin{itemize}
    \item $\phi$ is parallel with respect to the Levi-Civita connection. 
    
    \item The holonomy of $g$ is a subgroup of $G_2$. 
    
    \item $g$ is Ricci-flat.  
\end{itemize}
\end{Le}

A $G_2$-manifold $(M,\phi)$ may carry $1$, $2$, $4$ or $8$ parallel spinors.
In the last three cases, $M$ is the product of a circle with a Calabi-Yau threefold, 
the product of a torus with a K3 surface or a 7-dimensional torus. M-theory on 
$\mathbb{R}^{3,1} \times M$ thus yields as its low-energy limit a four-dimensional 
field theory with $\mathcal{N}=2$, $4$, or $8$. If $(M,\phi)$ carries only one parallel 
spinor, the holonomy is either $G_2$, $SU(3)\rtimes \Delta$ or $Sp(1)\rtimes \Delta$, 
where $\Delta$ is a discrete group such that the holonomy group acts irreducibly
on the tangent space. In the last two of these cases, $M$ is a suitable quotient 
of a $G_2$-manifold with $2$ or $4$ parallel spinors. 

In the rest of this section, let $(M,\phi)$ be a $G_2$-manifold with exactly one
parallel spinor. We consider M-theory on $\mathbb{R}^{3,1}\times M$. In the
low-energy limit we obtain a four-dimensional theory with $\mathcal{N}=1$. Its
field content consists of $b^3(M)$ chiral multiplets, $b^2(M)$ abelian vector 
multiplets and the graviton multiplet. In particular, the behaviour of the particles 
with spin $1$ is described by a Yang-Mills theory with gauge group $U(1)^{b^2(M)}$.
Therefore, we cannot obtain a supersymmetric extension of the standard model of 
particle physics by this ansatz. Nevertheless, it is possible to obtain field theories 
with non-abelian gauge groups if we allow $M$ to be a $G_2$-orbifold. We need 
a particular kind of singularities that we describe below.

\begin{Th}
\label{ADE-Thm} (Felix Klein \cite{Klein})
Let $\Gamma$ be a finite subgroup of $SU(2)$ and let $\tau: SU(2) \rightarrow
SO(3)$ be the usual double cover. Then $\Gamma$ is conjugate either to a
cyclic group that is generated by

\begin{equation}
\label{An-1}
\left(\,
\begin{array}{cc}
\exp{\left(\frac{2\pi i}{n}\right)} & 0 \\
0 & \exp{\left(-\frac{2\pi i}{n}\right)} \\ 
\end{array}
\,\right)
\end{equation}

or it is up to conjugation the preimage of the dihedral, tetrahedal, octahedral or 
icosahedral subgroup of $SO(3)$ with respect to $\tau$.  
\end{Th}

\begin{Rem}
\begin{enumerate}
    \item In fact, Felix Klein classified the finite subgroups of
    $SL(2,\mathbb{C})$. Since for any finite $\Gamma\subset
    SU(2)$ there exists a $\Gamma$-invariant Hermitian form on
    $\mathbb{C}^2$, both problems are equivalent.  
    
    \item The finite subgroups of $SU(2)$ are often denoted as follows.

    \begin{itemize}
        \item The cyclic group with $n+1$ elements is $A_n$.
    
        \item The preimage of the embedding of the dihedral group
        with $2n-4$ elements into $SO(3)$ is $D_n$.
    
        \item The preimage of the tetrahedral group is $E_6$. 
    
        \item The preimage of the octahedral group is $E_7$.
    
        \item The preimage of the icosahedral group is $E_8$.  
    \end{itemize}
\end{enumerate}
\end{Rem}

These are the same names as of the simply laced Dynkin diagrams. There
is in fact a mathematical connection between the finite subgroups of
$SU(2)$ and the Dynkin diagrams that is known as the McKay correspondence
\cite{McKay}. Let $\Gamma\subset SU(2)$ be finite. The quotient 
$\mathbb{C}^2/\Gamma$ is a singular algebraic variety. A singularity of this 
kind is called \emph{du Val singularity} or \emph{ADE-singularity}. It is worth 
mentioning that the resolution graph of an ADE-singularity is precisely the 
Dynkin diagram that corresponds to $\Gamma$. We can extend our definition 
to singularities of complex orbifolds. 

\begin{Def}
An $n$-dimensional complex orbifold $M$ has an \emph{ADE-singula\-rity
of type $\Gamma$} along a complex $(n-2)$-dimensional submanifold
$N$ if for all $p\in N$ there exists a neighborhood of $p$ that can
be biholomorphically identified with an open neighborhood of 
$0\in \mathbb{C}^{n-2} \times \mathbb{C}^2/\Gamma$.  
\end{Def}

The definition of a $G_2$-orbifold with ADE-singularities is a bit more
complicated since $G_2$-orbifolds have odd dimension and thus do not 
carry a complex structure. In order to properly define that term we need 
a further concept.

\begin{Def}
A 3-dimensional submanifold $N$ of a $G_2$-manifold $(M,\phi)$ is 
called \emph{associative} if the restriction of $\phi$ to $N$ is the 
same as the volume form of $N$. 
\end{Def}

Let $p$ be a point on an associative
submanifold $N$. The tangent space $T_pM$ splits into the
tangent space of $N$ and its normal space $V_p$. Any $g\in G_2$
that acts as the identity on $T_pN$ leaves $V_p$ invariant. 
The subgroup of all those $g$ is isomorphic to $SU(2)$ and
it acts by its 2-dimensional complex representation on $V_p$.
Therefore, $T_pN\times V_p/\Gamma$ where $\Gamma\subset
SU(2)$ is finite carries a well-defined 3-form with stabilizer $G_2$.
This motivates the following definition. 

\begin{Def}
\begin{enumerate}
    \item Let $M$ be a 7-dimensional orbifold and let $\phi$ be a smooth
    3-form on $M$. We assume that $\phi_p$ is stabilized by $G_2$ if $p$
    is a smooth point of $M$. If $p$ is not smooth, $T_pM$ is isomorphic
    to $\mathbb{R}^7/\Gamma$ where $\Gamma\subset GL(7,\mathbb{R})$
    is finite. Let $\pi: \mathbb{R}^7 \rightarrow T_pM$ be the quotient map.
    We assume that $\pi^{\ast}\phi_p$ is stabilized by $G_2$, too. If 
    additionally $\phi$ is closed and coclosed, we call $(M,\phi)$ a 
    \emph{$G_2$-orbifold}. 

    \item Let $(M,\phi)$ be a $G_2$-orbifold, $\Gamma$ be a finite subgroup 
    of $SU(2)$ and let $N$ be an associative submanifold of $M$. We identify
    $\mathbb{C}^2\times \mathbb{R}^3$ canonically with $\mathbb{R}^7$.
    Let $\phi_0$ be a 3-form on $\mathbb{C}^2/\Gamma \times \mathbb{R}^3$
    whose pull-back to $\mathbb{R}^7$ with respect to the quotient map is
    stabilized by $G_2$. We assume that for any $p\in N$ there exists a 
    neighborhood $U$ with $p\in U\subset M$ and a coordinate system $f:U
    \rightarrow V \subset \mathbb{C}^2/\Gamma \times \mathbb{R}^3$ with 
    $f(p)=0$ and $f^{\ast}\phi_0 = \phi_p$. In this situation, we say that
    \emph{$(M,\phi)$ has an ADE-singularity of type $\Gamma$ along $N$}. 
\end{enumerate}    
\end{Def}

We describe very briefly the role of $G_2$-orbifolds in M-theory. For a more
detailed account we refer the reader to \cite{Ach,Ach2,Morrison1,Morrison2}. 
It is believed that M-theory is well-defined on 11-dimensional spacetimes with 
ADE-singularities. Before we consider compactifications on $G_2$-mani\-folds 
with ADE-singularities, we take a look at M-theory on $\mathbb{C}^2/\Gamma
\times \mathbb{R}^{6,1}$. Let $G$ be the compact simple Lie group whose 
Dynkin diagram has the same name as $\Gamma$. If we neglect gravity, 
the physics on $\{0\}\times\mathbb{R}^{6,1}$ is described by a super 
Yang-Mills theory with gauge group $G$ in the low-energy limit. For each 
$\Gamma$ there exists a one-parameter family of smooth hyper-K\"ahler 
metrics $(g_t)_{t\in(0,\infty)}$ on an underlying manifold $X$ such that any 
$g_t$ is asymptotic to $\mathbb{C}^2/\Gamma$ with the flat metric and 
$(X,g_t)$ converges for $t\rightarrow 0$ to $\mathbb{C}^2/\Gamma$
in the Gromov-Hausdorff limit \cite{Kro1}. M-theory on $(X,g_t)\times 
\mathbb{R}^{6,1}$ yields a super Yang-Mills theory with gauge group 
$U(1)^{\text{rank}(G)}$.   

Next, let $M$ be a $G_2$-manifold with an ADE-singularity of type $\Gamma$ 
along an associative submanifold $N$. The simplest case is that $b^1(N)=0$ and
that the singularity has no monodromy, i.e. there exists a tubular neighborhood
of $N$ that is diffeomorphic to $U/\Gamma \times N$ where $U$ is an open disc
in $\mathbb{C}^2$. In this situation, a subgroup $U(1)^{\text{rank}(G)}$ of the gauge
group gets enhanced to $G$. If the monodromy is non-trivial, it may act as an outer
automorphism on $\Gamma$. This automorphism induces a symmetry transformation
of the Dynkin diagram of $G$ and thus an outer automorphism $\tau$ of $G$. Instead
of a Yang-Mills theory with gauge group $G$ we have the subgroup of all fixed points
of $\tau$ as gauge group. If $b^1(N)>0$, the gauge group is left unchanged but 
there are $b^1(N)$ additional chiral multiplets in the massless spectrum.  

The reason why we have $b^3(M)$ chiral multiplets in the smooth case is that 
the moduli space of a compact $G_2$-manifold has dimension $b^3(M)$ \cite{Joyce}. 
An analogous result for $G_2$-orbifolds is not yet proven. Therefore, it is not
entirely clear if we obtain $b^3(M)$ chiral multiplets if $M$ is a $G_2$-orbifold. 
Nevertheless, we will compute the third Betti number of our examples since it is 
an important topological invariant.

\section{Flat examples of $G_2$-orbifolds}
\label{Local}

In this section, we describe the $G_2$-orbifolds from \cite{Ach} in detail. Before
we start, we recall the definition of a hyper-K\"ahler manifold.

\begin{Def}
A \emph{hyper-K\"ahler manifold} is a $4n$-dimensional Riemannian
manifold $(M,g)$ together with three linearly independent integrable 
complex structures $I^1$, $I^2$ and $I^3$ such that 

\begin{enumerate}
    \item the complex structures satisfy the quaternion multiplication 
    relation $I^1 I^2 I^3 = -\text{Id}$ and
    
    \item the 2-forms $\omega^1$, $\omega^2$ and $\omega^3$
    defined by $\omega^i(X,Y) = g(I^i(X), Y)$ are K\"ahler forms.   
\end{enumerate}

The data $(g,\omega^1,\omega^2,\omega^3)$ on $M$ are called the 
\emph{hyper-K\"ahler structure}. 
\end{Def}

\begin{Rem}
\begin{enumerate}
    \item The holonomy of the metric on a $4n$-dimensional hyper-K\"ahler 
    manifold is a subgroup of $Sp(n)$. If $n=1$ and $M$ is compact, it is either 
    a four-dimensional torus or a K3 surface. 
    
    \item The set of all parallel complex structures on a hyper-K\"ahler 
    manifold is a sphere $S^2$. 
\end{enumerate}
\end{Rem}

Let $T^3 := \mathbb{R}^3/\mathbb{Z}^3$ be the three-dimensional torus with 
coordinates $x^1$, $x^2$, $x^3$ and the flat metric $g_{ij}:= \delta_{ij}
dx^i \wedge dx^j$. Furthermore, let $S$ be a four-dimensional hyper-K\"ahler 
manifold. The three K\"ahler-forms $\omega^1$, $\omega^2$ and $\omega^3$ 
satisfy $\omega^i \wedge \omega^j = \delta^{ij} \text{vol}_S$, where 
$\text{vol}_S$ is the volume form of $S$. We define the following 3-form 
on $S\times T^3$:

\begin{equation}
\phi := \omega^1\wedge dx^1 + \omega^2\wedge dx^2 +
\omega^3\wedge dx^3 + dx^1\wedge dx^2\wedge dx^3\:.
\end{equation}

$\phi$ is a parallel $G_2$-structure on $S\times T^3$. For any $p\in S$
the submanifold $\{p\}\times T^3$ is associative. The Hodge-dual of
$\phi$ is

\begin{equation}
\ast\phi =\text{vol}_S + \omega^1\wedge dx^2\wedge dx^3
+ \omega^2\wedge dx^3\wedge dx^1 + \omega^3\wedge dx^1\wedge dx^2
\end{equation}

We allow $S$ to be an orbifold, too. Any singular point shall have a neighborhood 
that looks like an open set $U\subset \mathbb{C}^2/\Gamma$ where $0\in U$ and
$\Gamma$ is a finite subgroup of $SU(2)$. We have to require 
$\Gamma$ to be a subgroup of the holonomy group $Sp(1)$ in order to make the
hyper-K\"ahler structure on $S$ well-defined. Let $H$ be a finite group that acts 
freely and isometrically on $T^3$. The action of any $h\in H$ can be written as 

\begin{equation}
x \mathbb{Z}^3 \mapsto (A^h x + v^h) \mathbb{Z}^3  
\end{equation} 

where $v^h\in\mathbb{R}^3$ and $A^h\in SO(3)$. We assume that the action of 
$H$ can be extended to a free action on $S\times T^3$ such that the 
pull-back of $h$ acts as 

\begin{equation}
\label{PullBackCond}
\omega^i \mapsto A^h_{ij} \omega^j\:.
\end{equation}

If this is the case, $H$ leaves $\phi$ invariant and the quotient $M:=(S \times T^3)/H$
carries a well-defined $G_2$-structure. Since $H$ acts freely, there are no singularities 
worse than the ADE-singularities of $S$. We assume that there is no one-dimensional 
subspace of $\mathbb{R}^3$ that is invariant under the Euclidean motions $x\mapsto 
A^h x + v^h$. The holonomy is a group of type $Sp(1)\rtimes \Delta$, where $\Delta$
is a subgroup of $H$. Since the holonomy acts irreducibly on the tangent space, it is 
not a subgroup of $SU(3)$ although it is a subgroup of $G_2$. Therefore, $M$ carries 
exactly one parallel spinor.  

$S$ is chosen in \cite{Ach} as $\mathbb{C}^2/\Gamma$ together with the
canonical flat $Sp(1)$-structure, where  $\Gamma$ is generated by 

\begin{equation}
\left(\,
\begin{array}{cc}
\exp{\left(\frac{2\pi i}{n}\right)} & 0 \\
0 & \exp{\left(-\frac{2\pi i}{n}\right)} \\ 
\end{array}
\,\right)
\end{equation}

In other words, we have a singularity of type $A_{n-1}$. We introduce 
the following maps $\alpha,\beta,\gamma:\mathbb{C}^2\times T^3
\rightarrow \mathbb{C}^2\times T^3$:

\begin{eqnarray}
\alpha(z_1,z_2,x_1,x_2,x_3) & := & (\exp{\left(\tfrac{2\pi i}{n}\right)}z_1, 
\exp{\left(-\tfrac{2\pi i}{n}\right)}z_2, x_1,x_2,x_3) \\  
\beta(z_1,z_2,x_1,x_2,x_3) & := & (-z_1,z_2,-x_1 + \tfrac{1}{2},-x_2,x_3+\tfrac{1}{2})\\  
\gamma(z_1,z_2,x_1,x_2,x_3) & := & (-\overline{z_1},-\overline{z_2},-x_1,x_2+\tfrac{1}{2},-x_3)   
\end{eqnarray}

where $z_1$ and $z_2$ are the complex coordinates on $\mathbb{C}^2$
and we have abbreviated cosets of type $v\mathbb{Z}^3$ by $v$.
$\alpha$ describes the action of $\Gamma$ on 
$\mathbb{C}^2 \times T^3$. $\beta$ and $\gamma$ generate a group
$H_1$ that is isomorphic to $\mathbb{Z}_2\times \mathbb{Z}_2$. 
$\Gamma$ is a normal subgroup of the group that is generated by $\alpha$, 
$\beta$ and $\gamma$. The action of $H_1$ on $\mathbb{C}^2/\Gamma 
\times T^3$ is thus well-defined. By a straightforward calculation we see
that the pull-backs $\beta^{\ast}$ and $\gamma^{\ast}$ act on a suitable
basis $(\omega^1,\omega^2,\omega^3)$ of self-dual forms on $\mathbb{C}^2$
as described by equation (\ref{PullBackCond}). For all $k\in\mathbb{Z}$ 
we have 

\begin{equation}
\label{MonodromyRelation}
\gamma \alpha^k \gamma^{-1} = \alpha^{-k}\:.
\end{equation}

The singularities of $(\mathbb{C}^2/\Gamma \times T^3)/H_1$ therefore 
have a non-trivial monodromy if $n\geq 2$. There is a second class of 
examples in \cite{Ach}. $\alpha$ is as before but we divide by a group 
$H_2$ that is generated by the maps 

\begin{eqnarray}
\beta'(z_1,z_2,x_1,x_2,x_3) & := & (-z_1,z_2,-x_1,-x_2 + \tfrac{3}{4},x_3+\tfrac{1}{2})\\  
\eta(z_1,z_2,x_1,x_2,x_3) & := & (-i\overline{z_2},i\overline{z_1},x_1+\tfrac{1}{4},-x_2+\tfrac{1}{4},-x_3)   
\end{eqnarray}  

$H_2$ is isomorphic to $\mathbb{Z}_4 \rtimes \mathbb{Z}_2$. Again we can 
show that the action of $H_2$ on $\mathbb{C}^2/\Gamma \times T^3$ is 
well-defined and that the pull-backs of $\beta'$ and $\eta$ satisfy 
(\ref{PullBackCond}). Any element of $H_2$ commutes with $\alpha$ and 
the singularities of the quotient thus have trivial monodromy.  

As we will see below, we still obtain a well-defined action of $H_1$ and $H_2$
on $\mathbb{C}^2/\Gamma \times T^3$ if we replace $\Gamma$ by a group 
of type $D_k$ ($k\geq 4$) or $E_k$ ($k\in\{6,7,8\}$). We restrict the generators 
of $H_1$ and $H_2$ to $\mathbb{C}^2$ and obtain maps $\tau_1,\tau_2,\tau_3:
\mathbb{C}^2\rightarrow\mathbb{C}^2$ with

\begin{eqnarray}
\tau_1(z_1,z_2) & = & (-z_1,z_2) \\
\tau_2(z_1,z_2) & = & (-\overline{z_1},-\overline{z_2}) \\
\tau_3(z_1,z_2) & = & (-i\overline{z_2},i\overline{z_1}) 
\end{eqnarray}

Since we have

\begin{equation}
\tau_i SU(2) \tau_i^{-1} = SU(2)
\end{equation}

for $i\in \{1,2,3\}$, the action of $H_1$ and $H_2$ is indeed well-defined. 
The definition of any map $\beta$, $\gamma$, $\beta'$ or $\eta$ contains
a shift by a non-integer, for example $x_1\mapsto x_1+\tfrac{1}{4}$ in the case of 
$\eta$. Therefore, the projection of the group action to $T^3$ does not have
any fixed point. This ensures that the quotient does not have singularities 
apart from the ADE-singularity and that $T^3/H_1$ as well as $T^3/H_2$ are 
smooth manifolds.  

Choosing $S$ as $\mathbb{C}^2/\Gamma$ has the advantage that the group action  
as well as the hyper-K\"ahler structure on the quotient manifold $M$ can be written down 
explicitly. The disadvantage is that $M$ is not compact. Therefore, it would be nice 
to find compact examples which are in a certain sense similar to those from \cite{Ach}. 
More precisely, we search for the following objects:

\begin{enumerate}
    \item a K3 surface $S$ with ADE-singularities together with
    
    \item three complex structures $I^1$, $I^2$, $I^3$ and three
    K\"ahler forms $\omega^1$, $\omega^2$ and $\omega^3$
    which make $S$ a hyper-K\"ahler manifold and 
    
    \item isometries $\rho_1:S\rightarrow S$, $\rho_2:S\rightarrow S$
    with $\rho_1^2=\rho_2^2=\text{Id}$
\end{enumerate}

such that either

\begin{equation}
\label{RhoCondition1}
\begin{array}{lll}
\rho_1^{\ast}\omega^1 = -\omega^1 &
\rho_1^{\ast}\omega^2 = -\omega^2 &
\rho_1^{\ast}\omega^3 = \omega^3 \\
\rho_2^{\ast}\omega^1 = -\omega^1 &
\rho_2^{\ast}\omega^2 = \omega^2 &
\rho_2^{\ast}\omega^3 = -\omega^3 \\
\end{array}
\end{equation}

or

\begin{equation}
\label{RhoCondition2}
\begin{array}{lll}
\rho_1^{\ast}\omega^1 = -\omega^1 &
\rho_1^{\ast}\omega^2 = -\omega^2 &
\rho_1^{\ast}\omega^3 = \omega^3 \\
\rho_2^{\ast}\omega^1 = \omega^1 &
\rho_2^{\ast}\omega^2 = -\omega^2 &
\rho_2^{\ast}\omega^3 = -\omega^3 \\
\end{array}
\end{equation}

If we find such isometries, we can redefine $\beta$ as 

\begin{equation}
\begin{array}{rcl}
\beta:S\times T^3 & \rightarrow & S\times T^3 \\
\beta(p,x_1,x_2,x_3) & := & (\rho_1(p),-x_1+\tfrac{1}{2},-x_2,x_3+\tfrac{1}{2}) \\
\end{array}
\end{equation}

where $\rho_1$ shall satisfy (\ref{RhoCondition1}). 
The other maps $\gamma$, $\beta'$ and $\eta$ can be redefined analogously 
and we obtain an action of $H_1$ or $H_2$ on $S\times T^3$. The quotients
$(S\times T^3)/H_1$ and $(S\times T^3)/H_2$ are compact $G_2$-orbifolds with
ADE-singularities. Their holonomy group is in both cases $Sp(1)\rtimes
\mathbb{Z}_2^2$.

\section{The K3 moduli space}

Before we prove the existence of the involutions $\rho_1$ and $\rho_2$ 
from the previous section, we need some facts about K3 surfaces and 
their moduli space. The results of this section are well-known. We refer 
the reader to \cite[Chapter VIII]{BHPV}, \cite[Chapter 7.3]{Joyce} and 
references therein for a more detailed account.

\begin{Def}
A \emph{K3 surface} is a compact, simply connected, complex surface 
with trivial canonical bundle. 
\end{Def}

\begin{Rem}
At the moment, we require a K3 surface to be smooth. Later on, we 
include K3 surfaces with ADE-singularities into our considerations.
\end{Rem}

All K3 surfaces are deformations of each other. The diffeomorphism 
type of the underlying four-dimensional real manifold is therefore 
unique. In particular, their cohomology ring is fixed, too. 

\begin{Th}
Let $S$ be a K3 surface.
\begin{enumerate}
    \item The Hodge numbers of $S$ are determined by 
    $h^{0,0}(S)= h^{2,0}(S)=1$, $h^{1,0}(S)=0$, and $h^{1,1}=20$.  

    \item The second integer cohomology $H^2(S,\mathbb{Z})$ 
    together with the intersection form is a lattice that is isomorphic to 

    \begin{equation}
    L := 3 H \oplus 2(-E_8) \:,
    \end{equation}

    where $H$ is the hyperbolic plane lattice with the bilinear form

    \begin{equation}
    \label{Hyperbolic}
    \left(\,
    \begin{array}{cc}
    0 & 1 \\
    1 & 0 \\
    \end{array}
    \,\right)
    \end{equation}

   and $-E_8$ is the root lattice of $E_8$ together with the negative of the usual 
   bilinear form.
\end{enumerate}   
\end{Th}

\begin{Def}
\begin{enumerate}
    \item The lattice $L$ from the above theorem is called the \emph{K3 lattice}.
    
    \item A K3 surface $S$ together with a lattice isometry $\phi:H^2(S,\mathbb{Z})
    \rightarrow L$ is called a \emph{marked K3 surface}.
    
    \item Two marked K3 surfaces $(S,\phi)$ and $(S',\phi')$ are called
    \emph{isomorphic} if there exists a biholomorphic map $f:S\rightarrow S'$
    such that $\phi\circ f^{\ast} = \phi'$. 
    
    \item The \emph{moduli space of marked K3 surfaces $\mathcal{M}_{K3}$}
    is the set of all marked K3 surfaces modulo isomorphisms.
\end{enumerate}
\end{Def}

Since the canonical bundle of $S$ is trivial, there exists a global holomorphic 
$(2,0)$-form on $S$. We denote it by $\omega^2 + i\omega^3$, 
where $\omega^2$ and $\omega^3$ are real 2-forms. We will see later that
$\omega^2$ and $\omega^3$ are the same objects as in Section \ref{Local},
in other words they are K\"ahler forms for appropriate complex structures on 
$S$. We denote the intersection form by a dot. It is easily seen that

\begin{equation}
[\omega^2 + i\omega^3]\cdot[\omega^2 + i\omega^3] = 0\:,
\qquad 
[\omega^2 + i\omega^3]\cdot[\omega^2 - i\omega^3] > 0\:,
\end{equation}

where $[\eta]\in H^2(S,\mathbb{C})$ denotes the cohomology class of a 
2-form $\eta$. Let $\mathbb{K}\in\{\mathbb{R},\mathbb{C}\}$, $L_{\mathbb{K}}:= 
L \otimes \mathbb{K}$ and $\phi_{\mathbb{K}}: 
H^2(S,\mathbb{K}) \rightarrow L_{\mathbb{K}}$ be the $\mathbb{K}$-linear 
extension of a marking $\phi$. Our  considerations motivate the following definition.

\begin{Def}
\begin{enumerate}
    \item We denote the complex line that is spanned by $x\in L_{\mathbb{C}}$
    by $\ell_x$. The set 
    
    \begin{equation}
    \Omega:=\{ \ell_x \in \mathbb{P}(L_{\mathbb{C}}) | x\cdot x = 0,\: x\cdot\overline{x}>0 \}  
    \end{equation}
    
    is called the \emph{period domain}. 
    
    \item Let $(S,\phi)$ be a marked K3 surface. The complex line spanned by 
    $\phi_{\mathbb{C}} ([\omega^2 + i\omega^3])$, where $[\omega^2 + i\omega^3]$
    is the cohomology class of the $(2,0)$-form, defines a point $p(S,\phi)\in 
    \Omega$ called the \emph{period point}. This assignment defines a map
    $p:\mathcal{M}_{K3}\rightarrow\Omega$, called the \emph{period map}. 
\end{enumerate}
\end{Def} 

In order to describe $\mathcal{M}_{K3}$ some further definitions are necessary.

\begin{Def}
\begin{enumerate}
    \item Let $S$ and $S'$ be K3 surfaces. A lattice isometry 
    $\psi:H^2(S,\mathbb{Z}) \rightarrow H^2(S',\mathbb{Z})$ is called a 
    \emph{Hodge-isometry} if its $\mathbb{C}$-linear extension preserves 
    the Hodge decomposition $H^2(S,\mathbb{C})=H^{2,0}(S) \oplus H^{1,1}(S) 
    \oplus H^{0,2}(S)$. 

    \item A class $x\in H^2(S,\mathbb{Z})$ is called \emph{effective} if
    there exists an effective divisor $D$ with $c_1(\mathcal{O}_X(D))
    = x$. 

   \item The connected component of the set $\{x\in H^{1,1}(S) | x\cdot x > 
   0\}$ which contains a K\"ahler class is called the \emph{positive
   cone} of $S$.

    \item A Hodge-isometry $\psi:H^2(S,\mathbb{Z}) \rightarrow 
    H^2(S',\mathbb{Z})$ is called \emph{effective} if it maps the 
    positive cone of $S$ to the positive cone of $S'$ and effective 
    classes in $H^2(S,\mathbb{Z})$ to effective classes in 
    $H^2(S',\mathbb{Z})$. 
\end{enumerate}
\end{Def}

\begin{Rem}
The restriction of the intersection form to $H^{1,1}(S)$
has signature $(1,19)$. The set $\{x\in H^{1,1}(S) | x\cdot x > 0\}$ thus
has exactly two connected components. It is known that any K3 surface is
K\"ahler. Therefore exactly one of the connected components contains
K\"ahler classes and the definition of the positive cone makes sense.
\end{Rem}

The following lemma fits into this context and will be helpful later on.

\begin{Le} (\cite[p. 313]{BHPV})
\label{EffectiveLemma}
Let $S$ and $S'$ be K3 surfaces and $\psi:H^2(S,\mathbb{Z})
\rightarrow H^2(S',\mathbb{Z})$ be a Hodge-isometry.  If $\psi$
maps at least one K\"ahler class of $S$ to a K\"ahler class of
$S'$, then $\psi$ is effective.
\end{Le}

With help of the terms that we have defined above we are able to
state the following theorems. 

\begin{Th} (Torelli Theorem)
Let $S$ and $S'$ be two unmarked K3 surfaces. If there exists an effective
Hodge-isometry $\psi:H^2(S',\mathbb{Z}) \rightarrow H^2(S,\mathbb{Z})$,
$\psi$ is the pull-back of a biholomorphic map $f:S\rightarrow S'$. 
\end{Th}

\begin{Th} 
The period map $p:\mathcal{M}_{K3}\rightarrow\Omega$ is surjective. 
\end{Th}

These two theorems make an explicit description of $\mathcal{M}_{K3}$
possible. Since we do not need that description in this article, we refer the 
reader to \cite{BHPV,Joyce} for details. What is important is that 
$\mathcal{M}_{K3}$ is a smooth complex manifold of dimension $20$. 
$\mathcal{M}_{K3}$ is not Hausdorff, but the moduli space that we are 
really interested in is.

\begin{Def}
\begin{enumerate}
    \item A \emph{marked pair} is a pair of a marked K3 surface and a 
    K\"ahler class on it. We usually write a marked pair as $(S,\phi,y)$
    where $(S,\phi)$ is a marked K3 surface and $y\in H^{1,1}(S)$
    is a K\"ahler class.
    
    \item Two marked pairs $(S,\phi,y)$ and $(S',\phi',y')$
    are called \emph{isomorphic} if there exists a biholomorphic map 
    $f:S\rightarrow S'$ that satisfies $\phi\circ f^{\ast} = \phi'$ and 
    $f^{\ast}y' = y$. 
    
    \item The \emph{moduli space of marked pairs $\mathcal{M}^p_{K3}$}
    is the set of all marked pairs modulo isomorphisms.
\end{enumerate}
\end{Def}

We define the following two sets:  

\begin{equation}
\begin{array}{rcl}
K\Omega & := & \{ (\ell_x,y)\in \Omega \times L_{\mathbb{R}} | \text{Re}(x)\cdot y
= \text{Im}(x)\cdot y = 0, y\cdot y > 0 \} \\
(K\Omega)^0 & := & \{ (\ell_x,y)\in K\Omega | y\cdot d\neq 0 \:\: \forall d\in L
\:\:\text{with}\:\: d^2 = -2, x\cdot d=0 \} 
\end{array}
\end{equation}

and the \emph{refined period map} 

\begin{eqnarray}
p':\mathcal{M}^p_{K3} & \rightarrow & \Omega \times L_{\mathbb{R}} \\
p'(S,\phi,y) & := & (p(S,\phi),\phi_{\mathbb{R}}(y))
\end{eqnarray}

\begin{Th}
$p'$ takes its values in $(K\Omega)^0$. Moreover, it is a bijection 
between $\mathcal{M}^p_{K3}$ and $(K\Omega)^0$. As a consequence, 
$\mathcal{M}^p_{K3}$ is a real analytic Hausdorff ma\-nifold of
dimension $60$. 
\end{Th}

The following lemma shows why $\mathcal{M}^p_{K3}$ is the
moduli space that we need. 

\begin{Le}
Let $M$ be the underlying real manifold of a K3 surface. There is a one-to-one 
correspondence between hyper-K\"ahler metrics on $M$ together with a
choice of a parallel complex structure and a marking on the one hand
and marked pairs on the other hand.
\end{Le}

\begin{proof}
Let $g$ be a hyper-K\"ahler metric on $M$. We choose one of the parallel
complex structures on $(M,g)$ and denote it by $I$. $(M,I)$ is a K3 surface
$S$. Furthermore, we choose an arbitrary marking $\phi$. Let $\omega^I$ be
the K\"ahler form on $M$ that is skew-Hermitian with respect to $g$ and $I$. 
$(S,\phi,[\omega^I])$ is a marked pair.

Conversely, let $(S,\phi,y)$ be a marked pair and $I$ be the complex structure 
on $S$. The Calabi-Yau theorem guarantees that there exists a unique K\"ahler 
form $\omega\in y$ such that $g(X,Y):=\omega(X,I(Y))$ is a Ricci-flat K\"ahler 
metric. Since $S$ is simply connected, the holonomy of $g$ is $SU(2)$. 
$SU(2)$ is isomorphic to $Sp(1)$ and the metric $g$ is actually
hyper-K\"ahler.  
\end{proof}

Again, let $M$ be the underlying manifold of a K3 surface. The moduli space 
of all hyper-K\"ahler metrics on $M$ together with a marking is diffeomorphic to

\begin{equation}
\begin{aligned}
\Omega_{hyp}:=\{ & (x_1,x_2,x_3)\in L_{\mathbb{R}}^3 | 
x_1^2 = x_2^2 = x_3^2>0,\: x_1\cdot x_2 = x_1\cdot x_3 = x_2\cdot x_3 =0, \\
& \!\! \not\exists\: d\in L \:\:\text{with}\:\: d^2 = -2 \:\:\: \text{and} \:\:\:
x_1\cdot d = x_2\cdot d = x_3\cdot d=0    
\} / SO(3) \\
\end{aligned}
\end{equation}

where $SO(3)$ acts by rotations on the three-dimensional space that is spanned by
$x_1$, $x_2$ and $x_3$, see for example \cite[p. 161]{Joyce}. 
If we do not want to fix a marking, we divide this space by $\text{Aut}(L)$, where 
$\text{Aut}(L)$ is the automorphism group of the lattice $L$. The moduli space of K3 
surfaces with a hyper-K\"ahler metric that may admit ADE-singularities can be described
even simpler. We define 

\begin{equation}
\begin{aligned}
\Omega'_{hyp} := \{ & (x_1,x_2,x_3)\in L_{\mathbb{R}}^3 | 
x_1^2 = x_2^2 = x_3^2>0\:, x_1\cdot x_2 = x_1\cdot x_3 = x_2\cdot x_3=0
\} / SO(3) \\
\end{aligned}
\end{equation}

and the \emph{hyper-K\"ahler period map} 

\begin{equation}
p'_{hyp}(g,\phi) := (\phi_{\mathbb{R}}([\omega^1]),\phi_{\mathbb{R}}([\omega^2]),
\phi_{\mathbb{R}}([\omega^3]))SO(3)  
\end{equation}

where $g$ is a hyper-K\"ahler metric on $M$, $\phi$ is a marking, $\omega^i$
with $i\in\{1,2,3\}$ are the forms that define the hyper-K\"ahler structure and 
$SO(3)$ acts on the span of $\phi_{\mathbb{R}}([\omega^1])$, 
$\phi_{\mathbb{R}}([\omega^2])$ and $\phi_{\mathbb{R}}([\omega^3])$.
$p'_{hyp}$ is a diffeomorphism between the moduli space of marked 
hyper-K\"ahler metrics on $M$ together with their degenerations to metrics 
with ADE-singularities and $\Omega'_{hyp}$. Since the intersection form
has signature $(3,19)$, $\Omega'_{hyp}$ admits a natural action by $SO(3,19)$. 
Moreover, it is diffeomorphic to the symmetric space

\begin{equation}
SO(3,19)/(SO(3)\times SO(19)) \:.
\end{equation}

Let $\alpha:=(x,y,z)SO(3) \in \Omega'_{hyp}\setminus\Omega_{hyp}$. We 
interpret the corresponding point in the moduli space geometrically, cf.
\cite[p. 161 - 162]{Joyce}. We define 

\begin{equation}
\label{SingSet}
\mathcal{D}_{\alpha}:= \{ d\in L | d^2=-2, x_1\cdot d = x_2\cdot d = 
x_3\cdot d = 0 \}
\end{equation}

Any $d\in\mathcal{D}_{\alpha}$ can be identified with an element of the homology group
$H_2(S)$. More precisely, the homology class contains a unique minimal 2-sphere 
with self-intersection number $-2$. Its area $A$ can be calculated in terms 
of the intersection form. More precisely, we have $A^2=\sum_{i=1}^3 
(d\cdot [\omega^i])^2$. If we move from a generic point in $\Omega'_{hyp}$ 
towards a point where $\mathcal{D}$ consists of a single element, the 2-sphere 
collapses to a point. In other words, we obtain a rational double point. 

Let the cardinality of $\mathcal{D}$ be $>1$. By joining $d_1,d_2\in\mathcal{D}$ 
by $d_1\cdot d_2$ edges, we obtain a graph $G$. $G$ is the disjoint union of 
simply laced Dynkin diagrams. As we approach $\alpha$, a set of 2-spheres whose 
intersection numbers are given by $d_i\cdot d_j$ collapses, which means that the 
Dynkin diagrams describe the type of the singularities. For example, if $G$ consists 
of one Dynkin diagram of type $E_8$ and $8$ isolated points, the singularities of the 
K3 surface are at $9$ different points. At one of them we have a singularity of type 
$E_8$ and at the other ones we have rational double points. 

Our aim is to construct K3 surfaces with a certain kind of automorphisms. This can 
be done with help of the following theorem.

\begin{Th}
\label{IsomThm}
Let $S$ be a K3 surface (possibly with ADE-singularities) together with a 
hyper-K\"ahler metric $g$ and K\"ahler forms $\omega^1$, $\omega^2$ and 
$\omega^3$. Moreover, let $V\subset H^2(S,\mathbb{R})$ be the subspace that 
is spanned by $[\omega^1]$, $[\omega^2]$ and $[\omega^3]$.  

\begin{enumerate}
    \item Let $f:S\rightarrow S$ be an isometry of $g$. The pull-back
    $f^{\ast}:H^2(S,\mathbb{Z}) \rightarrow H^2(S,\mathbb{Z})$ is
    an isometry of the lattice $H^2(S,\mathbb{Z})$. Its $\mathbb{R}$-linear
    extension preserves $V$.
    
    \item \label{Isom} Let $\psi:H^2(S,\mathbb{Z}) \rightarrow H^2(S,\mathbb{Z})$
    be a lattice isometry such that $\psi_{\mathbb{R}}(V)=V$. Moreover,
    $\psi_{\mathbb{C}}$ shall preserve the positive cone. Then there exists 
    an isometry $f:S\rightarrow S$ such that $f^{\ast} = \psi$. 
    
    \item Let $f:S\rightarrow S$ be an isometry that acts as the identity
    on $H^2(S,\mathbb{Z})$. Then, $f$ itself is the identity map. As a 
    consequence, the isometry from (\ref{Isom}) is unique.  
\end{enumerate}
\end{Th}

\begin{proof}
\begin{enumerate}
    \item $f^{\ast}$ preserves the intersection form and thus is a 
    lattice isometry. Moreover, $f$ leaves the vector space of all 
    parallel 2-forms invariant. Since this space is spanned by $\omega^1$, 
    $\omega^2$ and $\omega^3$, $f^{\ast}$ preserves $V$. 

    \item We split $H^2(S,\mathbb{C})$ into $\psi_{\mathbb{C}}(H^{2,0}(S))$,
    $\psi_{\mathbb{C}}(H^{0,2}(S))$ and the orthogonal complement of those 
    two subspaces. This splitting is a Hodge-structure on $S$ and $\psi$ is a 
    Hodge-isometry between $S$ together with the original Hodge-structure 
    and $S$ with the new Hodge-structure. \newline

    There are two complex structures on $S$. One of them is induced by the 
    period point $\ell_{\phi([\omega^2 + i\omega^3])}$ and the other one
    by $\psi_{\mathbb{C}}(\ell_{\phi([\omega^2 + i\omega^3])})$,
    where $\phi$ is a fixed making of $S$.  Since $[\omega^1]$ is a 
    K\"ahler class wit respect to the first complex structure, $\psi([\omega^1])$
    is a K\"ahler class with respect to the second one. According to Lemma 
    \ref{EffectiveLemma}, $\psi$ is an effective Hodge-isometry.  It follows 
    from the Torelli theorem - which also holds for K3 surfaces with ADE-singularities
    - that there exists a map $f:S\rightarrow S$ that is 
    biholomorphic with respect to the two complex structures. \newline

    There exists a unique $\eta$ that is a $(1,1)$-form with respect to the 
    second complex structure, determines a Ricci-flat K\"ahler metric and
    satisfies $[\eta] = (f^{\ast})^{-1} [\omega^1]$. Since $(f^{-1})^{\ast} \omega^1$
    has the same properties as $\eta$, we have $\eta = (f^{-1})^{\ast} \omega^1$. The
    second complex structure together with $\eta$ determines a metric
    $h$ on $S$ that is hyper-K\"ahler with K\"ahler forms $f^{\ast}\omega^k$ 
    for $k=1,2,3$. We have $h=f^{\ast}g$ since there exists only 
    one hyper-K\"ahler metric with K\"ahler forms $f^{\ast}\omega^k$. 
    $f^{\ast}$ acts as an element of $SO(3)$ on the space that is spanned
    by the complex structures $I^k$. Therefore, $h$ and $g$ have the same 
    sphere of parallel integrable complex structures. It follows that $h=g$ and we
    finally have proven that $f^{\ast} g = g$. 

    \item This follows from Proposition 11.3 in Chapter VIII in \cite{BHPV}.  
\end{enumerate} 
\end{proof}

\begin{Rem}
If we had omitted the condition that $\psi_{\mathbb{C}}$ preserves the positive
cone, the above theorem would have been slightly more complicated. In that 
situation $\psi:=-\text{Id}_{H^2(S,\mathbb{Z})}$ satisfies all condition from the 
theorem. The isometry of $f:S\rightarrow S$ that we would obtain would be
the identity map, but it would have to be interpreted as an antiholomorphic map
between $(S,I^1)$ and $(S,-I^1)$. The additional sign is necessary to map the
K\"ahler form $\omega^1$ to $-\omega^1$.   
\end{Rem}

\section{Compact examples of $G_2$-orbifolds}

In order to construct our $G_2$-orbifolds, we need to find a K3 surface with a 
hyper-K\"ahler metric and two isometries $\rho_1$ and $\rho_2$ that satisfy 
either (\ref{RhoCondition1}) or (\ref{RhoCondition2}). A K3 surface
with a hyper-K\"ahler structure is specified by $3$ elements $x_1$, $x_2$ and
$x_3$ of $L_{\mathbb{R}}$ that satisfy certain conditions. Since we want to 
distinguish the different summands of $L$, we write

\begin{equation}
L = H^1 \oplus H^2 \oplus H^3 \oplus (-E_8)^1 \oplus (-E_8)^2\:.
\end{equation}

We choose for each $H^i$ a basis $(v_1^i,v_2^i)$ such that the bilinear form
on $H^i$ has the standard form (\ref{Hyperbolic}). Let  $x_i := v_1^i + 2v_2^i$. 
It is easy to see that $x_i\cdot x_j= 4\delta_{ij}$. Therefore,
$(x_1,x_2,x_3)SO(3)\in \Omega'_{hyp}$. It follows from Theorem \ref{IsomThm}
that in order to find $\rho_1$ and $\rho_2$ it suffices to find lattice isometries 
$\psi_1$ and $\psi_2$ preserving the positive cone such that either

\begin{equation}
\begin{array}{lll}
\psi_1(x_1) = -x_1 &
\psi_1(x_2) = -x_2 &
\psi_1(x_3) = x_3 \\
\psi_2(x_1) = -x_1 &
\psi_2(x_2) = x_2 &
\psi_2(x_3) = -x_3 \\
\end{array}
\end{equation}

or

\begin{equation}
\begin{array}{lll}
\psi_1(x_1) = - x_1 &
\psi_1(x_2) = -x_2 &
\psi_1(x_3) = x_3 \\
\psi_2(x_1) = x_1 &
\psi_2(x_2) = -x_2 &
\psi_2(x_3) = -x_3 \\
\end{array}
\end{equation}

The maps that act as the identity on $(-E_8)^1$ and $(-E_8)^2$ 
and as plus or minus the identity on the $H^i$ are obviously lattice
isometries. If we carefully look how the signs affect the complex structure
$I^1$ and the form $\omega^1$, we see that a map of this kind preserves 
the positive cone if and only if it is minus the identity on an even 
number of $H^i$s. By choosing the signs appropriately, we see that 
$\psi_1$ and $\psi_2$ exist. The set (\ref{SingSet}) that describes 
the singularities of our K3 surface $S$ is

\begin{equation}
\mathcal{D} := \{d\in (-E_8)^1 | d^2 = -2\} \cup \{d\in (-E_8)^2 | d^2 = -2\}
\end{equation}

since the complement of $x_i$ in $H^i$ is spanned by $-v_1^i + 2v_2^i$
which has length $-4$. $\mathcal{D}$ is the disjoint union of two root systems of
$E_8$ and $S$ therefore has two singular points with $E_8$-singularities. 
We are also interested in finding K3 surfaces with milder singularities. 
This can be done by the following method. We replace $x_1$ by 
$v_1^1 + 2v_2^1 + u^1 + u^2$ where the $u^i$s are vectors in 
$(-E_8)^i\otimes \mathbb{R}$. If the length of the $u^i$
is sufficiently small, $x_1$ still has a positive length $\ell$. By
replacing $x_2$ by $\frac{\ell}{4} x_2$ and $x_3$ by $\frac{\ell}{4} x_3$, we
obtain $x_1^2=x_2^2=x_3^2=\ell$ and $x_1\cdot x_2 = x_1\cdot x_3 =
x_2\cdot x_3 = 0$. Therefore, we have $(x_1,x_2,x_3)SO(3)\in \Omega'_{hyp}$.
The set $\mathcal{D}$ becomes 

\begin{equation}
\underbrace{\{d\in (-E_8)^1 | d^2 = -2,d\cdot u_1=0\}}_{=:\mathcal{D}_1} \cup
\underbrace{\{d\in (-E_8)^2 | d^2 = -2,d\cdot u_2= 0\}}_{=:\mathcal{D}_2}
\end{equation} 

Let $\{\alpha_1,\ldots,\alpha_8\}$ be a set of simple roots of $E_8$. 
We choose a proper subset of $\{\alpha_1,\ldots,\alpha_8\}$ and complement it
to a linearly independent family with seven elements by vectors
in $(-E_8)^1\otimes\mathbb{R}$ whose components are irrational. We define $u_1$ 
as a vector that is orthogonal to all elements of this family. $\mathcal{D}_1$ is a 
root system whose simple roots are our chosen subset. We obtain the Dynkin 
diagram of $\mathcal{D}_1$ by deleting a subset of nodes of $E_8$. If the 
Dynkin diagram of $\mathcal{D}_1$ is connected, it can be any element 
of the following set:

\begin{equation}
\{0,A_1,\ldots,A_7,D_4,D_5,D_6,D_7,E_6,E_7,E_8 \} \:.
\end{equation}

The cases where the Dynkin diagram is not connected are of course
allowed, too. The vector $u_2$ can be defined analogously. Geometrically
this procedure correspond to a partial crepant resolution of the 
$E_8$-singularities. All in all, we have constructed two families of 
$G_2$-orbifolds that have a wide range of different singularities. 
We call the quotients of $S\times T^3$ by the group that is generated by 
$\beta$ and $\gamma$ \emph{of the first kind} and those by $\beta'$ and 
$\eta$ \emph{of the second kind}. 

Let $H$ be the group by which we divide $S\times T^3$. $H$ acts on
both factors separately and our $G_2$-orbifold $M$ is thus a fiber bundle 
over $T^3/H$ with K3 fibers. The singular set $N$ of $M$ consists of a finite 
number of sections of that bundle. This number equals the number $k$ of 
singular points of $S$. $N$ is therefore diffeomorphic to the disjoint union 
of $k$ copies of $T^3/H$. Our next step is to compute the invariants $b^2(M)$,
$b^3(M)$ and $b^1(N)$ as well as the monodromies of the singularities since 
they are relevant for the physics of M-theory compactified on $M$. 

The de Rham cohomology $H^2(M)$ is isomorphic to the $H$-invariant 
part of $H^2(S\times T^3)$. The K\"unneth formula yields

\begin{equation}
H^2(S\times T^3) = \left(H^2(S)\otimes H^0(T^3)\right) \oplus 
\left(H^1(S)\otimes H^1(T^3)\right) \oplus 
\left(H^0(S)\otimes H^2(T^3)\right) 
\end{equation}   

We determine the $\beta$-invariant part of $H^2(S\times T^3)$. 
Since $\beta$ is an involution, the eigenvalues of 
$\beta^{\ast}$ are at most $-1$ and $1$. We mark the eigenspaces and
the corresponding Betti numbers with a subscript $1$ or $-1$. Since 
$H^0(S)$ and $H^0(T^3)$ are $\beta$-invariant and $H^1(S)$ is 
trivial, we have  

\begin{equation}
\begin{aligned}
H^2_1(S\times T^3) \cong H^2_1(S) \oplus H^2_1(T^3) \:.
\end{aligned}
\end{equation} 

If $S$ is a smooth K3 surface, $H^2_1(S)$ is $\left(H^3 \oplus (-E_8)^1 
\oplus (-E_8)^2\right)\otimes\mathbb{R}$ and $H^2_1(T^3)$ is spanned 
by $dx^{12}$. Therefore, we have $b^2_1(S\times T^3)= 19$. The same 
reasoning can be applied to $\gamma$. This time $H^2_1(S)$ is $\left(H^2 
\oplus (-E_8)^1 \oplus (-E_8)^2\right)\otimes \mathbb{R}$ and $H^2_1(T^3)$ 
is spanned by  $dx^{13}$. The $H$-invariant part of $H^2(S\times T^3)$ is 
the intersection of the $\beta$- and the $\gamma$-invariant  part. $H^2(M)$ 
thus is isomorphic to $\left((-E_8)^1 \oplus (-E_8)^2\right)\otimes
\mathbb{R}$ and we have $b^2(M)=16$. If $S$ has singularities, we 
have $b^2(M)=16 - \text{rank}(G)$ since $\text{rank}(G)$ two-spheres 
collapse. If $M$ is of the second kind, we obtain $b^2(M)=16  - 
\text{rank}(G)$ by analogous arguments. Our next step is to determine 
$b^3(M)$. Analogously as above we have

\begin{equation}
\begin{aligned}
H^3(S\times T^3) =\: & \left(H^3(S)\otimes H^0(T^3)\right) \oplus 
\left(H^2(S)\otimes H^1(T^3)\right)
\oplus \left(H^1(S)\otimes H^2(T^3)\right) \\
& \oplus \left(H^0(S)\otimes H^3(T^3)\right)\\
H^3_1(S\times T^3) \cong\: & \{0\} \oplus \left(H^2_1(S)\otimes H^1_1(T^3)\right)
\oplus \left(H^2_{-1}(S)\otimes H^1_{-1}(T^3)\right) \\
& \oplus \{0\} \oplus H^3(T^3)\\
\end{aligned}
\end{equation}  

where the eigenspaces are with respect to $\beta^{\ast}$. The non-trivial
spaces in the above equation are   

\begin{equation}
\begin{array}{rcl}
H^2_1(S) & = & \left(H^3 \oplus (-E_8)^1 \oplus (-E_8)^2\right)\otimes\mathbb{R} \\
H^1_1(T^3) & = & \text{span}(dx^3) \\
H^2_{-1}(S) & = & \left(H^1 \oplus H^2\right)\otimes \mathbb{R} \\
H^1_{-1}(T^3) & = & \text{span}(dx^1,dx^2) \\  
H^3(T^3) & = & \text{span}(dx^{123}) \\
\end{array}
\end{equation}

If we consider $\gamma^{\ast}$ instead of $\beta^{\ast}$, we have the same
decomposition of $H^3_1(S\times T^3)$ and

\begin{equation}
\begin{array}{rcl}
H^2_1(S) & = & \left(H^2 \oplus (-E_8)^1 \oplus (-E_8)^2\right)\otimes\mathbb{R} \\
H^1_1(T^3) & = & \text{span}(dx^2) \\
H^2_{-1}(S) & = & \left(H^1 \oplus H^3\right)\otimes \mathbb{R} \\
H^1_{-1}(T^3) & = & \text{span}(dx^1,dx^3) \\  
H^3(T^3) & = & \text{span}(dx^{123}) \\
\end{array}
\end{equation}

The intersection of the $\beta$- and the $\gamma$-invariant part of 
$H^3(S\times T^3)$ is spanned by $dx^{123}$ and all forms of
type $dx^i \wedge \alpha$ where $i\in \{1,2,3\}$ and $\alpha \in
H^i$. Therefore, we have $b^3(M)=7$. As before, we also have
$b^3(M)=7$ if $M$ is of the second kind. The two-spheres that 
collapse in the singular case all have homology classes that are 
Poincar\'{e} dual to elements of $(-E_8)^1 \oplus (-E_8)^2$. Therefore, 
the Betti number is the same in the smooth and in the singular case. 
Since there are no harmonic one-forms on the torus that are 
preserved by both generators of $H$, we have $b^1(N)=0$. 

Finally, we determine the monodromy group of the singularities. 
The reason for the non-trivial monodromy of the examples
from \cite{Ach} is that the group $\Gamma$ by which we divide 
$\mathbb{C}^2$ does not commute with $H$. Our situation
is different since we do not obtain $S$ by a quotient construction. 

Let $p\in S$ be a singular point. We show that $p$ is a fixed point of
$\rho_1$ and $\rho_2$. Since the $\rho_i$ are isometries, they have 
to map singular points to singular points of the same kind. $\rho_i$
may map $p$ to a singular point $q\neq p$ only if $p$ and
$q$ are of the same kind. If we resolve the singularity at $q$, we obtain 
a new isometry $\tilde{\rho}_i$ on the resolved K3 surface. $\tilde{\rho}_i$ 
maps a singular point to a smooth point which is impossible. 

The singularity at $p$ can be obtained by the contraction of certain curves
whose cohomology classes $d$ satisfy $d^2=-2$. We have constructed
$S$ in such a way that any $d$ is an element of $(-E_8)^2$ on which 
$\rho_1$ and $\rho_2$ act trivially. Therefore, $\rho_1$ and $\rho_2$ 
act trivially on the set of those curves in the resolved K3 surface, too.
In the limit where the curves are contracted, a neighborhood of $p$
becomes diffeomorphic to a neighborhood $U$ of $0\in
\mathbb{C}^2/\Gamma$. In this limit, $\rho_1$ and $\rho_2$ commute
with any element of $\Gamma$. 

Let $M$ be a $G_2$-orbifold of the first kind, $\pi:S\times T^3\rightarrow M$
be the projection map, $p\in M$ be a singular point and $q\in S\times
T^3$ be a point with $\pi(q)=p$. Moreover, let $c_1$ be a path from $q$ to 
$\beta(q)$ and $c_2$ be a path from $q$ to $\gamma(q)$. The monodromy 
group is generated by the monodromy along the loops $\pi\circ c_1$ and 
$\pi\circ c_2$.  Since the projections of $\beta$ and $\gamma$ to $S$ are 
$\rho_1$ and $\rho_2$, the monodromy along the loops is trivial. If $M$ is 
of the second kind, we can apply analogous arguments and we thus have 
proven that the monodromy is trivial in both cases. 

All in all, M-theory compactified on our orbifolds yields a 
four-dimen\-sional field theory whose bosonic part contains a 
Yang-Mills theory with gauge group $U(1)^{16-\text{rank}(G)} \times 
G$, where $G$ is the group that corresponds to the singularities, and the
graviton. Moreover, we suspect that there are $b^3(M)=7$ complex 
scalar fields. It is an interesting coincidence that the maximal gauge 
group that we obtain by this method is $E_8\times E_8$, which is one 
of the possible gauge groups of heterotic string theory.

\end{document}